\documentclass[11pt]{article}
\usepackage{amsmath}
\usepackage{amssymb}
\usepackage{amsthm}
\usepackage[usenames]{color}
\usepackage{amscd}
\usepackage{indentfirst}

\usepackage[colorlinks=true,linkcolor=blue,filecolor=red,
citecolor=webgreen]{hyperref}
\definecolor{webgreen}{rgb}{0,.5,0}

\def\N{{\Bbb N}}
\def\Z{{\Bbb Z}}

\def\F{{\Bbb F}}

\def\1{{\bf 1}}

\newcommand{\DOT}{\text{\rm\Huge{.}}}

\def\id{\operatorname{id}}
\def\a{\boldsymbol{a}}

\newtheorem{theorem}{Theorem}
\newtheorem{lemma}[theorem]{Lemma}
\newtheorem{cor}[theorem]{Corollary}

\newtheorem{prop}[theorem]{Proposition}
\newtheorem{remark}[theorem]{Remark}

\begin{document}

\title{{\bf Counting Solutions of Quadratic Congruences in Several Variables Revisited}}
\author{L\'aszl\'o T\'oth}
\date{}
\maketitle

\centerline{Journal of Integer Sequences, Vol. 17 (2014), Article 14.11.6}

\begin{abstract} Let $N_k(n,r,\a)$ denote the number of incongruent
solutions of the quadratic congruence $a_1x_1^2+\ldots
+a_kx_k^2\equiv n$ (mod $r$), where $\a=(a_1,\ldots,a_k)\in \Z^k$,
$n\in \Z$, $r\in \N$. We give short direct proofs for certain less
known compact formulas on $N_k(n,r,\a)$, valid for $r$ odd, which go
back to the work of Minkowski, Bachmann and Cohen. We also deduce
some other related identities and asymptotic formulas which do not seem to appear in the
literature.
\end{abstract}

\medskip
{\sl 2010 Mathematics Subject Classification}:  11D79, 11A25, 11N37

{\it Key Words and Phrases}: quadratic congruence in many variables,
number of solutions, Jacobi symbol, Ramanujan sum, character sum,
Gauss quadratic sum, asymptotic formula

\section{Introduction}

Let $k$ and $n$ be positive integers and let $r_k(n)$ denote the number of
representations of $n$ as a sum of $k$ squares. More exactly,
$r_k(n)$ is the number of solutions $(x_1,\ldots,x_k)\in \Z^k$ of
the equation
\begin{equation} \label{r_k_def}
x_1^2+ \cdots +x_k^2 = n.
\end{equation}

The problem of finding exact formulas or good estimates for $r_k(n)$
and to study other related properties is one of the most fascinating
problems in number theory. Such results were obtained by several
authors, including Euler, Gauss, Liouville, Jacobi, Legendre and
many others. Some of these results are now well known and are
included in several textbooks. See, e.g., Grosswald \cite{Gro1985},
Hardy and Wright \cite[Ch.\ XX]{HarWri2008}, Hua \cite[Ch.\
8]{Hua1982}, Ireland and Rosen \cite[Ch.\ 17]{IreRos1990}, Nathanson
\cite[Ch.\ 14]{Nat2000}. See also Dickson \cite[Ch.\ VI--IX,
XI]{Dic1966}.

For example, one has the next exact formulas. Let $n$ be of
the form $n=2^{\nu} m$ with $\nu \ge 0$ and $m$ odd. Then
\begin{equation} \label{r_2}
r_2(n)= 4\sum_{d\mid m} (-1)^{(d-1)/2},
\end{equation}
\begin{equation} \label{r_4}
r_4(n)= 8(2+(-1)^n) \sum_{d\mid m} d.
\end{equation}

Exact formulas for $r_k(n)$ are known also for other values of $k$.
These identities are, in general, more complicated for $k$ odd
than in the case of $k$ even.

Now consider the equation \eqref{r_k_def} in the ring $\Z/r\Z$ of
residues (mod $r$), where $r$ is a positive integer. Equivalently, consider the
quadratic congruence
\begin{equation} \label{cong_def}
x_1^2+ \cdots + x_k^2 \equiv n \quad \text{ (mod $r$)},
\end{equation}
where $n\in \Z$. Let $N_k(n,r)$ denote the number of incongruent
solutions $(x_1,\ldots,x_k)\in \Z^k$ of \eqref{cong_def}. The
function $r\mapsto N_k(n,r)$ is multiplicative. Therefore, it is sufficient to
consider the case $r=p^s$, a prime power. Identities for $N_k(n,p^s)$ can be derived using
Gauss and Jacobi sums. For example, we refer to the explicit formulas for $N_k(0,p^s)$ given in
\cite[p.\ 46]{BEW1998} and for $N_k(1,p)$ given in \cite[Prop.\ 8.6.1]{IreRos1990}.
See also Dickson \cite[Ch.\ X]{Dic1966} for historical remarks.

Much less known is that for $k$ even and $r$ odd, $N_k(n,r)$ can be
expressed in a compact form using Ramanujan's sum. Furthermore, for
$k$ odd, $r$ odd and $\gcd(n,r)=1$, $N_k(n,r)$ can be given in terms of the
M\"obius $\mu$ function and the Jacobi symbol. All these formulas are
similar to \eqref{r_2} and \eqref{r_4}. Namely, one has the following identities:

0) For $k\equiv 0$ (mod $4$), $r$ odd, $n\in \Z$:
\begin{equation} \label{k_0}
N_k(n,r)= r^{k-1} \sum_{d\mid r} \frac{c_d(n)}{d^{k/2}}.
\end{equation}

1) For $k\equiv 1$ (mod $4$), $r$ odd, $n\in \Z$, $\gcd(n,r)=1$:
\begin{equation} \label{k_1}
N_k(n,r)= r^{k-1} \sum_{d\mid r} \frac{\mu^2(d)}{d^{(k-1)/2}} \left(
\frac{n}{d}\right).
\end{equation}

2) For $k\equiv 2$ (mod $4$), $r$ odd, $n\in \Z$:
\begin{equation} \label{k_2}
N_k(n,r)= r^{k-1} \sum_{d\mid r} (-1)^{(d-1)/2}\
\frac{c_d(n)}{d^{k/2}}.
\end{equation}

3) For $k\equiv 3$ (mod $4$), $r$ odd, $n\in \Z$, $\gcd(n,r)=1$:
\begin{equation} \label{k_3}
N_k(n,r)= r^{k-1} \sum_{d\mid r} (-1)^{(d-1)/2}\
\frac{\mu^2(d)}{d^{(k-1)/2}} \left( \frac{n}{d}\right).
\end{equation}

These are special cases of the identities deduced by Cohen in the
paper \cite{Coh1954Duke} and quoted later in his papers
\cite{Coh1963Nachr,Coh1964Collect,Coh1966AMM}. The proofs given in
\cite{Coh1954Duke} are lengthy and use the author's previous work,
although in Section 7 of \cite{Coh1954Duke} a direct approach using
finite Fourier sums is also described. According to Cohen the
formulas \eqref{k_0}--\eqref{k_3} are due in an implicit form by
Minkowski \cite[pp.\ 45--58, 166--171]{Min1911}. Cohen \cite[p.\ 27]{Coh1954Duke} says:
``We mention the work of Minkowski as an important example of the use of Fourier sums in
treating quadratic congruences. While Minkowski's approach was quite
general, his results were mainly of an implicit nature.'' Cohen \cite{Coh1966AMM}
refers also to the book of Bachmann \cite[Part 1, Ch.\ 7]{Bac1898}.

Another related compact formula, which seems to not appear in the literature is the following:
If $k\equiv 0$ (mod $4$), $r$ is odd and $n\in \Z$, then
\begin{equation} \label{k_0_mod 4_Jordan}
N_k(n,r)= r^{k/2-1} \sum_{d\mid \gcd(n,r)} d\, J_{k/2}(r/d),
\end{equation}
where $J_m$ is the Jordan function of order $m$.

It is the first main goal of the present paper to present short direct proofs
of the identities \eqref{k_0}--\eqref{k_0_mod 4_Jordan}. Slightly more
generally, we will consider --- as Cohen did --- the quadratic
congruence
\begin{equation} \label{cong_a}
a_1x_1^2+ \cdots +a_kx_k^2 \equiv n \quad \text{ (mod $r$)},
\end{equation}
where $n\in \Z$, $\a=(a_1,\ldots,a_k)\in \Z^k$ and derive formulas
for the number $N_k(n,r,\a)$ of incongruent solutions
$(x_1,\ldots,x_k)\in \Z^k$ of \eqref{cong_a}, assuming that $r$ is
odd. For the proofs we only need to express $N_k(n,r,\a)$ by a
trigonometric sum and to use the evaluation of the Gauss quadratic
sum. No properties concerning finite Fourier expansions or other
algebraic arguments are needed. The proof is quite simple if $k$ is
even and somewhat more involved if $k$ is odd.
We also evaluate $N_k(n,2^{\nu})$ ($\nu \in \N$) for certain values of $k$ and $n$
and consider some special cases of \eqref{cong_a}.

Our second main goal is to establish asymptotic
formulas --- not given in the literature, as far as we know --- for the sums $\sum_{r\le x} N_k(n,r)$, taken over all
integers $r$ with $1\le r\le x$, in the cases $(k,n)=(1,0), (1,1), (2,0), (2,1), (3,0), (3,1), (4,0), (4,1)$.
Similar formulas can be deduced also for other special choices of $k$ and $n$. Note that the mean
values of the functions $r\mapsto N_k(n,r)/n^{k-1}$ were investigated by Cohen \cite{Coh1963Nachr},
but only over the odd values of $r$.

We remark that a character free method to determine the number of solutions of
the equation $x^2+my^2=k$ in the finite field $\F_p$ ($p$ prime) was
presented in a recent paper by Girstmair \cite{Gir2013}.


\section{Notation} \label{sect_2}

Throughout the paper we use the following notation:
$\N:=\{1,2,\ldots\}$, $\N_0:=\{0,1,2,\ldots\}$; $e(x)=\exp(2\pi
ix)$; $\left( \frac{\ell}{r} \right)$ is the Jacobi symbol
($\ell,r\in \N$, $r$ odd), with the conventions $\left(
\frac{\ell}{1} \right)=1$ ($\ell \in \N$), $\left( \frac{\ell}{r}
\right)=0$ if $\gcd(\ell,r)>1$; $c_r(n)$ denotes Ramanujan's sum (see, e.g., \cite[Ch.\ 8]{Apo1976}, \cite[Ch.\ XVI]{HarWri2008})
defined as the sum of $n$-th powers of the primitive $r$-th roots of
unity, i.e.,
\begin{equation} \label{c_r(n)}
c_r(n)= \sum_{\substack{j=1\\ \gcd(j,r)=1}}^r e(jn/r) \quad (r,n\in \N),
\end{equation}
where $c_r(0)=\varphi(r)$ is Euler's function and $c_r(1)=\mu(r)$ is
the M\"obius function; $S(\ell,r)$ is the quadratic Gauss sum
defined by
\begin{equation} \label{Gauss_S}
S(\ell,r)= \sum_{j=1}^r e(\ell j^2/r) \quad (\ell,r \in \N, \gcd(\ell,r)=1).
\end{equation}

Furthermore, $*$ is the Dirichlet convolution of arithmetical functions; $\1$, $\id$ and $\id_k$ are the functions given by
$\1(n)=1$, $\id(n)=n$, $\id_k(n)=n^k$ ($n\in \N$); $\tau(n)$ is the number of divisors of $n$; $J_k=\mu * \id_k$
is the Jordan function of order $k$, $J_k(n)=n^k \prod_{p\mid n}(1-1/p^k)$
($n\in \N$), where $J_1=\varphi$. Also, $\psi_k=\mu^2 * \id_k$ is the generalized Dedekind function, $\psi_k(n)=n^k
\prod_{p\mid n}(1+1/p^k)$ ($n\in \N$); $\zeta$ is the Riemann zeta function; $\gamma$ stands for the Euler constant; $\chi=\chi_4$ is the
nonprincipal character (mod $4$) and
$G=L(2,\chi)\doteq 0.915956$ is the Catalan constant given by
\begin{equation} \label{Catalan_const}
G= \sum_{n=0}^{\infty} \frac{(-1)^n}{(2n+1)^2}= \prod_{p\equiv 1 \text{ (mod $4$)}} \left(1-\frac1{p^2} \right)^{-1} \prod_{p\equiv -1
\text{ (mod $4$)}} \left(1+\frac1{p^2} \right)^{-1}.
\end{equation}


\section{General results}

We evaluate $N_k(n,r,\a)$ using the quadratic Gauss sum $S(\ell,r)$ defined by \eqref{Gauss_S}.

\begin{prop} \label{Prop_3_1} For every $k,r\in \N$, $n\in \Z$, $\a=(a_1,\ldots,a_k)\in \Z^k$ we have
\begin{equation*}
N_k(n,r,\a)= r^{k-1} \sum_{d\mid r} \frac1{d^k} \sum_{\substack{\ell=1\\
(\ell,d)=1}}^d e(-\ell n/d) S(\ell a_1,d) \cdots S(\ell a_k,d).
\end{equation*}
\end{prop}

\begin{proof} As is well known (see e.g., \cite[Th.\ 1.31]{Nar1983}), the number of solutions of a congruence can be expressed
using the familiar identity
\begin{equation} \label{id_fam}
\sum_{j=1}^r e(jt/r)= \begin{cases} r, & \text{ if \ $r\mid t$;} \\  0, & \text{ if \ $r\nmid
t$.} \end{cases}
\end{equation}
valid for every $r\in \N, t\in \Z$. In our case we obtain
\begin{equation*}
N_k(n,r,\a)=\frac1{r} \sum_{x_1=1}^r \cdots \sum_{x_k=1}^r \sum_{j=1}^r e((a_1x_1^2+\cdots + a_kx_k^2-n)j/r),
\end{equation*}
that is,
\begin{equation} \label{N}
N_k(n,r,\a)=\frac1{r} \sum_{j=1}^r e(-jn/r) \sum_{x_1=1}^r
e(ja_1x_1^2/r) \cdots \sum_{x_k=1}^r e(ja_kx_k^2/r).
\end{equation}

By grouping the terms of \eqref{N} according to the values $(j,r)=d$
with $j=d\ell$, $(\ell,r/d)=1$, we obtain
\begin{equation} \label{N_1}
N_k(n,r,\a)=\frac1{r} \sum_{d\mid r} \sum_{\substack{\ell=1\\
(\ell,r/d)=1}}^{r/d} e(-\ell n/(r/d)) \sum_{x_1=1}^r e(\ell a_1
x_1^2/(r/d)) \cdots \sum_{x_k=1}^r e(\ell a_k x_k^2/(r/d)),
\end{equation}
where, as it is easy to see, for every $j\in \{1,\ldots,k\}$,
\begin{equation} \label{N_2}
\sum_{x_j=1}^r e(\ell a_j x_j^2/(r/d)) = d S(\ell a_j,r/d).
\end{equation}

By inserting \eqref{N_2} into \eqref{N_1} and by putting $d$ instead of $r/d$, we
are ready.
\end{proof}

\begin{prop} \label{Prop_3_2} Assume that $k,r\in \N$, $r$ is odd, $n\in \Z$ and
$\a=(a_1,\ldots, a_k)\in \Z^k$ is such that $\gcd(a_1 \cdots a_k,r)=1$. Then
\begin{equation} \label{eq_r_odd}
N_k(n,r,\a)= r^{k-1} \sum_{d\mid r} \frac{i^{k(d-1)^2/4}}{d^{k/2}} \left(
\frac{a_1\cdots a_k}{d} \right) \sum_{\substack{\ell=1\\
(\ell,d)=1}}^d \left(\frac{\ell}{d}\right)^k e(-\ell n/d).
\end{equation}
\end{prop}

\begin{proof} We use that for every $r$ odd and $\ell \in \N$ such that $\gcd(\ell,r)=1$,
\begin{equation} \label{S_eval}
S(\ell,r)= \begin{cases} \left( \frac{\ell}{r} \right) \sqrt{r}, \ \ & \text{ if \ $r \equiv 1$ (mod $4$);} \\
i\, \left( \frac{\ell}{r} \right) \sqrt{r}, \ \ & \text{ if \ $r \equiv -1$ (mod $4$),} \end{cases}
\end{equation}
cf., e.g., \cite[Th.\ 1.5.2]{BEW1998}, \cite[Th.\ 7.5.6]{Hua1982}.
Now the result follows immediately from Proposition \ref{Prop_3_1}.
\end{proof}

\begin{prop} \label{Prop_3_3} Assume that $k\in \N$, $r=2^{\nu}$ ($\nu \in \N$), $n\in \Z$ and
$\a=(a_1,\ldots, a_k)\in \Z^k$ is such that $a_1, \ldots, a_k$ are odd. Then
\begin{equation*}
N_k(n,2^{\nu},\a)= 2^{\nu(k-1)} \left( 1 + \sum_{t=1}^{\lfloor \nu/2 \rfloor} \frac1{2^{kt}} \sum_{\substack{\ell=1\\
\ell \, \text{\rm odd}}}^{2^{2t}} e(-\ell n/2^{2t})(1+i^{\ell a_1})\cdots (1+i^{\ell a_k}) \right.
\end{equation*}
\begin{equation*}
\left. + \sum_{t=1}^{\lfloor (\nu-1)/2 \rfloor} \frac1{2^{kt}} \sum_{\substack{\ell=1\\
\ell \, \text{\rm odd}}}^{2^{2t+1}} e(-\ell n/2^{2t+1} +\ell(a_1+\cdots +a_k)/8) \right).
\end{equation*}
\end{prop}

\begin{proof} Using Proposition \ref{Prop_3_1} and putting $d=2^s$,
\begin{equation*}
N_k(n,2^{\nu},\a)= 2^{\nu(k-1)} \sum_{s=0}^{\nu} \frac1{2^{ks}} \sum_{\substack{\ell=1\\
\ell \, \text{\rm odd}}}^{2^s} e(-\ell n/2^s) S(\ell a_1,2^s) \cdots S(\ell a_k,2^s).
\end{equation*}

We apply that for every $\ell$ odd,
\begin{equation*}
S(\ell,2^{\nu})= \begin{cases}
0, & \text{ if \ $\nu=1$;} \\
(1+i^{\ell})2^{\nu/2}, & \text{ if \ $\nu$ is even;} \\
2^{(\nu+1)/2}e(\ell/8), & \text{ if \ $\nu>1$ is odd,}
\end{cases}
\end{equation*}
cf., e.g., \cite[Th.\ 1.5.1, 1.5.3]{BEW1998}, \cite[Th.\
7.5.7]{Hua1982}. Separating the terms corresponding to $s=2t$ even
and $s=2t+1$ odd, respectively we obtain the given formula.
\end{proof}


\section{The case \texorpdfstring{$k$}{k} even, \texorpdfstring{$r$}{r} odd}

Suppose that $k$ is even and $r$ is odd. In this case we deduce for
$N_k(n,r,\a)$ formulas in terms of the Ramanujan sums.

\begin{prop} \label{prop_4_1} {\rm (\cite[Th.\ 11 and Eq.\ (5.2)]{Coh1954Duke})} Assume that $k=2m$ ($m\in \N$),
$r\in \N$ is odd, $n\in \Z$, $\a=(a_1,\ldots,a_k) \in
\Z^k$, $\gcd(a_1\cdots a_k,r)=1$. Then
\begin{equation*}
N_{2m}(n,r,\a)= r^{2m-1} \sum_{d\mid r} \frac{c_d(n)}{d^m} \left(
\frac{(-1)^m a_1\cdots a_{2m}}{d} \right).
\end{equation*}
\end{prop}

\begin{proof} This is a direct consequence of Proposition \ref{Prop_3_2}. For $k$ even the inner sum of \eqref{eq_r_odd} is
exactly $c_d(n)$, by its definition \eqref{c_r(n)}, and applying that $\left(\frac{-1}{d} \right)=(-1)^{(d-1)/2}$ the proof is complete.
\end{proof}

In the special case $k=2$, $a_1=1$, $a_2=-D$, $r$ odd, $\gcd(D,r)=\gcd(n,r)=1$ Proposition \ref{prop_4_1} was deduced by Rabin
and Shallit \cite[Lemma 3.2]{RabSha1986}.

\begin{cor} \label{cor_4_2} If $k=4m$ ($m\in \N$), $r\in \N$ is odd, $n\in \Z$ and
$a_1\cdots a_k=1$ (in particular $a_1=\cdots =a_k=1$), then
\begin{equation*}
N_{4m}(n,r,\a)= r^{4m-1} \sum_{d\mid r} \frac{c_d(n)}{d^{2m}}.
\end{equation*}
\end{cor}

\begin{cor} If $k=4m+2$ ($m\in \N_0$), $r\in \N$ is odd, $n\in \Z$ and
$a_1\cdots a_k=1$ (in particular $a_1=\cdots = a_k=1$), then
\begin{equation} \label{eq_4m+2}
N_{4m+2}(n,r,\a)= r^{4m+1} \sum_{d\mid r} (-1)^{(d-1)/2}
\frac{c_d(n)}{d^{2m+1}}.
\end{equation}
\end{cor}

Therefore, the identities \eqref{k_0} and \eqref{k_2} are proved. In particular, the next simple formulas are valid: for every $r$ odd,
\begin{equation} \label{N_2(0,r)_r_odd}
N_2(0,r)= r \sum_{d\mid r} (-1)^{(d-1)/2}\ \frac{\varphi(d)}{d},
\end{equation}
\begin{equation} \label{form_spec_Jordan}
N_4(1,r)= r^3 \sum_{d\mid r} \frac{\mu(d)}{d^2}= r J_2(r).
\end{equation}

\begin{remark} {\rm In the case $k$ even and $a_1=\cdots =a_k=1$ for the proof of Proposition \ref{prop_4_1}
it is sufficient to use the formula $S^2(\ell,r)=(-1)^{(r-1)/2}r$
($r$ odd, $\gcd(\ell,r)=1$) instead of the much deeper result
\eqref{S_eval} giving the precise value of $S(\ell,r)$.}
\end{remark}

In the case $k=4m$ and $a_1\cdots a_k=1$ the next representation holds as well (already given in \eqref{k_0_mod 4_Jordan} in the case
$a_1=\cdots = a_k=1$).

\begin{cor} \label{cor_4_5} If $k=4m$ ($m\in \N$), $r\in \N$ is odd, $n\in \Z$ and
$a_1\cdots a_k=1$ (in particular $a_1=\cdots = a_k=1$), then
\begin{equation} \label{4m_even}
N_{4m}(n,r,\a)= r^{2m-1} \sum_{d\mid \gcd(n,r)} d\, J_{2m}(r/d).
\end{equation}
\end{cor}

\begin{proof} We use Corollary \ref{cor_4_2} and apply that for every fixed $n$, $c_{\DOT}(n)=\mu * \eta_{\DOT}(n)$, where
$\eta_r(n)=r$ if $r\mid n$ and $0$ otherwise. Therefore,
\begin{align*}
N_{4m}(n,r,\a)
& = r^{2m-1} \sum_{d\mid r} c_d(n) (r/d)^{2m} \\
& = r^{2m-1} \left( c_{\DOT}(n)* \id_{2m}\right)(r) \\
& =  r^{2m-1} \left(\mu* \id_{2m} * \eta_{\DOT}(n) \right)(r) \\
& = r^{2m-1} \left(J_{2m} * \eta_{\DOT}(n) \right)(r) \\
& = r^{2m-1} \sum_{d\mid \gcd(n,r)} d\, J_{2m}(r/d).
\end{align*}
\end{proof}

\begin{remark} {\rm The identity \eqref{4m_even} shows that for every $r$ odd, the function $n\mapsto N_{4m}(n,r)$ is even (mod $r$).
We recall that an arithmetic function $n\mapsto f(n)$ is said to be even (mod $r$) if $f(n)=f(\gcd(n,r))$ holds for every $n\in \N$.
We refer to \cite{TotHau2011} for this notion. In fact, for every $k$ even the function $n\mapsto N_k(n,r,\a)$ is even (mod $r$), where $r$ is
a fixed odd number and $\gcd(a_1\cdots a_k,r)=1$, since according to Proposition \ref{prop_4_1}, $N_k(n,r,\a)$ is a linear combination of the
values $c_d(r)$ with
$d\mid r$. See also \cite{Coh1964Collect}.}
\end{remark}

A direct consequence of \eqref{4m_even} is the next result:

\begin{cor} If $k=4m$ ($m\in \N$), $r\in \N$ is odd, $n\in \Z$ such that $\gcd(n,r)=1$, then
\begin{equation*}
N_{4m}(n,r)= r^{2m-1} J_{2m}(r)= r^{4m-1} \prod_{p\mid r} \left(1-\frac1{p^{2m}}\right).
\end{equation*}
\end{cor}

Another consequence of Proposition \ref{prop_4_1} is the following identity, of which proof is similar to the proof of Corollary \ref{cor_4_5}:

\begin{cor} \label{cor_4_8} If $k=4m+2$ ($m\in \N_0$), $r\in \N$ is odd, $n\in \Z$ and
$a_1\cdots a_k=-1$, then
\begin{equation*}
N_{4m+2}(n,r,\a)= r^{2m} \sum_{d\mid \gcd(n,r)} d\, J_{2m+1}(r/d).
\end{equation*}
\end{cor}

In the case $r=p^\nu$ ($p>2$ prime) and for special choices of $k$
and $n$ one can deduce explicit formulas from the identities of
above. For example (as is well known):

\begin{cor} For every prime $p>2$ and every $n\in \N$,
\begin{equation*}
N_2(n,p) = \begin{cases} 2p-1, & \text{ if \ $p\mid n$, $p\equiv 1$ {\rm
(mod $4$);}} \\ 1, & \text{ if \ $p\mid n$, $p\equiv -1 $ {\rm (mod $4$);}} \\
p-1, & \text{ if \ $p\nmid n$,  $p\equiv 1$ {\rm (mod $4$);}} \\ p+1, &
\text{ if \ $p\nmid n$, $p\equiv -1$ {\rm (mod $4$),}}
\end{cases}
\end{equation*}
\begin{equation*}
N_2(n,p^2) = \begin{cases} p(p-1), & \text{ if \ $p\nmid n$, $p\equiv 1$ {\rm (mod $4$);}} \\
2p(p-1), & \text{ if \ $p\mid n$, $p^2\nmid n$, $p\equiv 1$ {\rm (mod $4$);}} \\
3p^2-2p, & \text{ if \ $p^2\mid n$, $p\equiv 1$, {\rm (mod $4$);}} \\
p(p+1), & \text{ if \ $p\nmid n$, $p\equiv -1$ {\rm (mod $4$);}} \\
0, & \text{ if \ $p\mid n$, $p^2\nmid n$,  $p\equiv -1$ {\rm (mod $4$);}} \\
p^2, & \text{ if \ $p^2\mid n$, $p\equiv -1$ {\rm (mod $4$).}}
\end{cases}
\end{equation*}
\end{cor}


\section{The case \texorpdfstring{$k$}{k} odd, \texorpdfstring{$r$}{r} odd}

Now consider the case $k$ odd, $r$ odd. In order to apply Proposition \ref{Prop_3_2} we need to evaluate
the character sum
\begin{equation*}
T(n,r)= \sum_{\substack{j=1\\\gcd(j,r)=1}}^r \left(\frac{j}{r}\right) e(jn/r).
\end{equation*}

\begin{lemma} \label{lemma_5_1} Let $r,n\in \N$, $r$ odd such that $\gcd(n,r)=1$.

i) If $r$ is squarefree, then
\begin{equation} \label{def_T}
T(n,r) = \begin{cases} \left( \frac{n}{r} \right) \sqrt{r}, & \text{ if \ $r \equiv 1$ {\rm (mod $4$);}} \\
i \left( \frac{n}{r} \right) \sqrt{r}, & \text{ if \ $r \equiv -1$ {\rm (mod $4$).}} \end{cases}
\end{equation}

ii) If $r$ is not squarefree, then $T(n,r)=0$.
\end{lemma}

\begin{proof} For every $r$ odd the Jacobi symbol $j\mapsto \left(\frac{j}{r}\right)$ is a real character (mod $r$) and
$T(n,r)= \left(\frac{n}{r}\right) T(1,r)$ holds if $\gcd(n,r)=1$. See, e.g., \cite[Ch.\ 7]{Hua1982}.

i) If $r$ is squarefree, then $j\mapsto \left(\frac{j}{r}\right)$ is
a primitive character (mod $r$). Thus, $T(1,r)= \sqrt{r}$ for
$\left( \frac{-1}{r} \right)=1$ and $T(1,r)=i \sqrt{r}$ for $\left(
\frac{-1}{r} \right)=-1$ (\cite[Th.\ 7.5.8]{Hua1982}), giving
\eqref{def_T}.

ii) We show that if $r$ is not squarefree, then $T(1,r)=0$. Here $r$ can be written as $r=p^2s$, where $p$ is a prime and by putting $j=ks+q$,
\begin{equation*}
T(1,r) =\sum_{q=1}^s \sum_{k=0}^{p^2-1} \left( \frac{ks+q}{r} \right) e((ks+q)/r),
\end{equation*}
where
\begin{equation*}
\left( \frac{ks+q}{r} \right) = \left( \frac{ks+q}{p^2} \right) \left( \frac{ks+q}{s} \right) = \left( \frac{q}{s} \right)
\end{equation*}
and deduce
\begin{equation*}
T(1,r) =\sum_{q=1}^s \left( \frac{q}{s} \right) e(q/(p^2s)) \sum_{k=0}^{p^2-1} e(k/p^2)=0,
\end{equation*}
since the inner sum is zero using \eqref{id_fam}.
\end{proof}

Note that properties of the sum $T(n,r)$, including certain orthogonality results were obtained by Cohen \cite{Coh1954Duke} using different arguments.

\begin{prop} {\rm (\cite[Cor.\ 2]{Coh1954Duke})} Assume that $k=2m+1$ ($m\in \N_0$), $r\in \N$ is odd, $n\in \Z$ such that $\gcd (n,r)=1$,
$\a=(a_1,\ldots,a_k) \in \Z^k$, $\gcd(a_1\cdots a_k,r)=1$. Then
\begin{equation*}
N_{2m+1}(n,r,\a)= r^{2m} \sum_{d\mid r} \frac{\mu^2(d)}{d^m} \left(\frac{(-1)^m na_1\cdots a_{2m+1}}{d} \right).
\end{equation*}
\end{prop}

\begin{proof} Apply Proposition \ref{Prop_3_2}. For $k$ odd the inner sum of \eqref{eq_r_odd} is $T(-n,d)$,
where $T(n,r)$ is given by \eqref{def_T}. Since $r$ is odd and
$\gcd(n,r)=1$, if $d\mid r$, then $d$ is also odd and $\gcd(d,r)=1$. We deduce by
Lemma \ref{lemma_5_1} that
\begin{align*}
N_k(n,r,\a)
& = r^{k-1} \sum_{d\mid r} \frac{i^{k(d-1)^2/4}}{d^{k/2}} \left(
\frac{a_1\cdots a_k}{d} \right) T(-n,d) \\
& = r^{k-1} \sum_{\substack{d\mid r\\ d \text{ squarefree}}} \frac{i^{k(d-1)^2/4}}{d^{k/2}}  \left(
\frac{a_1\cdots a_k}{d} \right) i^{(d-1)^2/4}\, \left( \frac{-n}{d} \right) \sqrt{d} \\
& = r^{k-1} \sum_{d\mid r} \frac{\mu^2(d)}{d^{(k-1)/2}}  i^{(k+1)(d-1)^2/4} \left(
\frac{-na_1\cdots a_k}{d} \right),
\end{align*}
which gives the result by evaluating the powers of $i$.
\end{proof}

\begin{cor} If $k=4m+1$ ($m\in \N_0$), $r\in \N$ is odd, $n\in \Z$, $\gcd(n,r)=1$ and
$a_1\cdots a_k=1$ (in particular if $a_1=\cdots =a_k=1$), then
\begin{equation*}
N_{4m+1}(n,r,\a)= r^{4m} \sum_{d\mid r} \frac{\mu^2(d)}{d^{2m}} \left(
\frac{n}{d} \right).
\end{equation*}
\end{cor}

\begin{cor} \label{cor_5_4} If $k=4m+3$ ($m\in \N_0$), $r\in \N$ is odd, $n\in \Z$, $\gcd(n,r)=1$ and
$a_1\cdots a_k=1$ (in particular if $a_1=\cdots =a_k=1$), then
\begin{equation*}
N_{4m+3}(n,r,\a) = r^{4m+2} \sum_{d\mid r} \frac{\mu^2(d)}{d^{2m+1}} (-1)^{(d-1)/2}
\left(\frac{n}{d} \right).
\end{equation*}
\end{cor}

This proves the identities \eqref{k_1} and \eqref{k_3}.

\begin{cor} If $k=4m+3$ ($m\in \N_0$), $r\in \N$ is odd, $n\in \Z$, $\gcd(n,r)=1$ and
$a_1\cdots a_k=-1$, then
\begin{equation*}
N_{4m+3}(n,r,\a) = r^{2m+1} \psi_{2m+1}(r).
\end{equation*}
\end{cor}

To prove the next result we need the evaluation of
\begin{equation*}
V(r) = T(0,r) = \sum_{\substack{j=1\\\gcd(j,r)=1}}^r
\left(\frac{j}{r}\right),
\end{equation*}
not given by Lemma \ref{lemma_5_1}.

\begin{lemma} \label{lemma_5_5} If $r\in \N$ is odd, then
\begin{equation*}
V(r) = \begin{cases} \varphi(r), &  \text{ if \ $r$ is a square;} \\
0, &  \text{ otherwise.} \end{cases}
\end{equation*}
\end{lemma}

\begin{proof} If $r=t^2$ is a square, then $\left(\frac{j}{r}\right)=\left(\frac{j}{t^2}\right)=1$ for every
$j$ with $\gcd(j,r)=1$ and deduce that $V(r)=\varphi(r)$.

Now assume that $r$ is not a square. Then, since $r$ is odd, there is a prime $p>2$ such
that $r=p^{\nu}s$, where $\nu$ is odd and $\gcd(p,s)=1$. First we
show that there exists an integer $j_0$ such that $(j_0,r)=1$ and
$\left(\frac{j_0}{r}\right)=-1$. Indeed, let $c$ be a quadratic
nonresidue (mod $p$) and consider the simultaneous congruences
$x\equiv c$ (mod $p$), $x\equiv 1$ (mod $s$). By the Chinese
remainder theorem there exists a solution $x=j_0$ satisfying
\begin{equation*}
\left(\frac{j_0}{r}\right)=\left(\frac{j_0}{p}\right)^{\nu} \left(\frac{j_0}{s}\right)= \left(\frac{c}{p}\right)^{\nu} \left(\frac{1}{s}\right)=
(-1)^{\nu}=-1,
\end{equation*}
since $\nu$ is odd. Hence
\begin{equation*}
V(r)= \sum_{\substack{j=1\\\gcd(j,r)=1}}^r \left(\frac{jj_0}{r}\right)= \sum_{\substack{j=1\\\gcd(j,r)=1}}^r
\left(\frac{j}{r}\right) \left(\frac{j_0}{r}\right)= - \sum_{\substack{j=1\\\gcd(j,r)=1}}^r \left(\frac{j}{r}\right)= -V(r),
\end{equation*}
giving that $V(r)=0$.
\end{proof}

\begin{prop} \label{Prop_5_7} {\rm (\cite[Cor.\ 1]{Coh1954Duke})} Assume that $k,r\in \N$ are odd, $n=0$ and $\a=(a_1,\ldots,a_k)
\in \Z^k$, $\gcd(a_1\cdots a_k,r)=1$.
Then
\begin{equation*}
N_k(0,r,\a)= r^{k-1} \sum_{d^2 \mid r} \frac{\varphi(d)}{d^{k-1}},
\end{equation*}
which does not depend on $\a$.
\end{prop}

\begin{proof} From Proposition \ref{Prop_3_2} we have
\begin{equation*}
N_k(0,r,\a)= r^{k-1} \sum_{d\mid r} \frac{i^{k(d-1)^2/4}}{d^{k/2}} \left( \frac{a_1\cdots a_k}{d} \right) V(d),
\end{equation*}
where $V(d)$ is given by Lemma \ref{lemma_5_5}. We deduce
\begin{equation*}
N_k(0,r,\a)= r^{k-1} \sum_{d^2\mid r} \frac{i^{k(d^2-1)^2/4}}{d^k} \left( \frac{a_1\cdots a_k}{d^2} \right)\varphi(d^2)=
r^{k-1} \sum_{d^2\mid r} \frac{\varphi(d^2)}{d^k}.
\end{equation*}
\end{proof}

\begin{remark} {\rm For all the results of this section it was assumed that $\gcd(n,r)=1$. See \cite{Coh1966AMM}
for certain special cases of $\gcd(n,r)>1$.}
\end{remark}


\section{The case \texorpdfstring{$k$}{k} even, \texorpdfstring{$r=2^{\nu}$}{r=2nu}}

In this section let $a_1=\cdots =a_k=1$.

\begin{prop} \label{Prop_6_1} If $k\in \N$ is even and $n\in \Z$ is odd, then $N_k(n,2)=2^{k-1}$ and for every $\nu \in \N$, $\nu \ge 2$,
\begin{equation*}
N_k(n,2^{\nu})= 2^{\nu(k-1)} \left(1-\frac1{2^{k/2-1}}\cos \left(\frac{k\pi}{4} + \frac{n\pi}{2}\right) \right).
\end{equation*}
\end{prop}

\begin{proof} We obtain from Proposition \ref{Prop_3_3} by separating the terms according to $\ell=4u+1$ and $\ell=4u+3$, respectively,
\begin{equation*}
N_k(n,2^{\nu})= 2^{\nu(k-1)} \left( 1 + \sum_{t=1}^{\lfloor \nu/2 \rfloor} \frac1{2^{kt}} A_t +
\sum_{t=1}^{\lfloor (\nu-1)/2 \rfloor} \frac1{2^{kt}}B_t \right),
\end{equation*}
where
\begin{align*}
A_t & = (1+i)^k \sum_{u=0}^{2^{2t-2}-1} e(-(4u+1)n/2^{2t}) + (1-i)^k \sum_{u=0}^{2^{2t-2}-1} e(-(4u+3)n/2^{2t}) \\
& = \left((1+i)^k e(-n/2^{2t})+(1-i)^k e(-3n/2^{2t})\right) \sum_{u=0}^{2^{2t-2}-1} e(-un/2^{2t-2}),
\end{align*}
and
\begin{align*}
B_t &= \sum_{u=0}^{2^{2t-1}-1} \left( e(-(4u+1)n/2^{2t+1} +(4u+1)k/8) + e(-(4u+3)n/2^{2t+1} +(4u+3)k/8) \right) \\
& = \left( e(k/8-n/2^{2t+1})+ e(3k/8-3n/2^{2t+1}) \right)  \sum_{u=0}^{2^{2t-1}-1} e(-un/2^{2t-1} +ku/2).
\end{align*}

Since $n$ is odd, $n/2^{2t-2}\notin \Z$ for every $t\ge 2$. It follows that $A_t=0$ for every $t\ge 2$. Also,
\begin{equation*}
A_1= (1+i)^k e(-n/4) + (1-i)^k e(-3n/4) = - 2^{k/2+1}
\cos(k\pi/4+n\pi/2).
\end{equation*}

Similarly, since $k$ is even and $n$ is odd, $k/2-n/2^{2t-1}\notin \Z$ for every $t\ge 1$. It follows that $B_t=0$ for every $t\ge 1$.
This completes the proof.
\end{proof}

\begin{cor} \label{cor_6_2} If $k=4m$ ($m\in \N$) and $n\in \Z$ is odd, then for every $\nu \in \N$,
\begin{equation*}
N_{4m}(n,2^{\nu})= 2^{\nu(4m-1)}.
\end{equation*}
\end{cor}

\begin{cor} \label{Cor_6_3} If $k=4m+2$ ($m\in \N_0$) and $n=2t+1\in \Z$ is odd, then $N_{4m+2}(n,2)=2^{4m+1}$ and for every $\nu \ge 2$,
\begin{equation*}
N_{4m+2}(n,2^{\nu})= 2^{\nu(4m+1)} \left(1+\frac{(-1)^{m+t}}{2^{2m}} \right).
\end{equation*}
\end{cor}

By similar arguments one can deduce from Proposition \ref{Prop_3_3}:

\begin{prop} \label{prop_6_4} {\rm (\cite[p.\ 46]{BEW1998})} If $k=4m$ ($m\in \N$) and $n=0$, then for every $\nu \in \N$,
\begin{equation*}
N_{4m}(0,2^{\nu})= 2^{\nu(4m-1)} \left(1+\frac{(-1)^m(2^{(\nu-1)(2m-1)}-1)}{2^{(\nu-1)(2m-1)}(2^{2m-1}-1)} \right).
\end{equation*}
\end{prop}

\begin{prop} \label{prop_6_5} {\rm (\cite[p.\ 46]{BEW1998})} If $k=4m+2$ ($m\in \N_0$) and $n=0$ for every $\nu \in \N$,
\begin{equation*}
N_{4m+2}(0,2^{\nu})= 2^{\nu(4m+1)}.
\end{equation*}
\end{prop}


\section{The case \texorpdfstring{$k$}{k} odd, \texorpdfstring{$r=2^{\nu}$}{r=2nu}}

Now let $k$ be odd, $r=2^{\nu}$, $a_1=\cdots =a_k=1$. By similar arguments as in the previous sections we have

\begin{prop} \label{Prop_7_1} {\rm (\cite[p.\ 46]{BEW1998})} If $k\in \N$ is odd and $n=0$, then for every $\nu \in \N$,
\begin{equation*}
N_k(0,2^{\nu})= 2^{\nu(k-1)} \left(1+\frac{(-1)^{(k^2-1)/8}\cdot (2^{(k-2)\lfloor \nu/2 \rfloor}-1)}
{2^{(k-2)\lfloor \nu/2 \rfloor - (k-3)/2}(2^{k-2}-1) } \right).
\end{equation*}
\end{prop}

Other cases can also be considered, for example:

\begin{prop} \label{Prop_7_2} If $k=4m+3$ and $n=4t+1$ ($m,n\in \N_0$), then for every $\nu \in \N$,
\begin{equation*}
N_{4m+3}(n,2^{\nu})= 2^{\nu(4m+2)} \left(1+\frac{(-1)^{m}}{2^{2m+1}} \right).
\end{equation*}
\end{prop}


\section{Special cases} \label{Sect_spec_cases}

In this section we consider some special cases and deduce asymptotic formulas for $k=1,2,3,4$ and $\a=(1,\ldots,1)$.

\subsection{The congruence \texorpdfstring{$x^2\equiv 0$}{x2=0} (mod \texorpdfstring{$r$}{r})}

For $k=1$ and $n=0$ we have the congruence $x^2\equiv 0$ (mod $r$). Its number of solutions, $N_1(0,r)$ is the sequence A000188 in \cite{OEIS}.
It is well known and can be deduced directly that $N_1(0,p^{\nu})=p^{\lfloor \nu/2 \rfloor}$ for every prime power $p^{\nu}$ ($\nu \in \N$).
This leads to the Dirichlet series representation
\begin{equation} \label{Dirichlet_8_1}
\sum_{r=1}^{\infty} \frac{N_1(0,r)}{r^s} = \frac{\zeta(2s-1)\zeta(s)}{\zeta(2s)}.
\end{equation}

Our next result corresponds to the classical asymptotic formula of Dirichlet
\begin{equation*}
\sum_{n\le x} \tau(n) = x\log x +(2\gamma-1)x + O(x^{1/2}).
\end{equation*}

\begin{prop} \label{prop_8_1} We have
\begin{equation} \label{error_2_3}
\sum_{r\le x} N_1(0,r)= \frac{3}{\pi^2} x\log x +c x + O(x^{2/3}),
\end{equation}
where $c=\frac{3}{\pi^2}\left(3\gamma-1-\frac{2\zeta'(2)}{\zeta(2)}\right)$.
\end{prop}

\begin{proof} By the identity \eqref{Dirichlet_8_1} we infer that for every $r\in \N$,
\begin{equation*}
N_1(0,r)=\sum_{a^2b^2c=r} \mu(a)b.
\end{equation*}

Using Dirichlet's hyperbola method we have
\begin{equation*}
E(x):= \sum_{b^2c\le x} b = \sum_{b\le x^{1/3}} b \sum_{c\le x/b^2} 1 + \sum_{c\le x^{1/3}} \sum_{b\le (x/c)^{1/2}} b - \sum_{b\le x^{1/3}} b
\sum_{c\le x^{1/3}} 1,
\end{equation*}
which gives by the trivial estimate (i.e., $|x- \lfloor x \rfloor|<1$),
\begin{equation*}
E(x)= \frac1{2}x\log x + \frac1{2}(3\gamma-1)x+ O(x^{2/3}).
\end{equation*}

Now,
\begin{equation*}
\sum_{r\le x} N_1(0,r)= \sum_{a\le x^{1/2}} \mu(a) E(x/a^2)
\end{equation*}
and easy computations complete the proof.
\end{proof}

\begin{remark} {\rm The error term of \eqref{error_2_3} can be improved by the method of exponential sums (see, e.g., \cite[Ch.\ 6]{Bor2012}).
Namely, it is $O(x^{2/3-\delta})$ for some explicit $\delta$ with $0<\delta <1/6$.}
\end{remark}

\subsection{The congruence \texorpdfstring{$x^2\equiv 1$}{x2=1} (mod \texorpdfstring{$r$}{r})}

It is also well known, that in the case $k=1$ and $n=1$ for the number of solutions of the congruence $x^2\equiv 1$ (mod $r$) one has
$N_1(1,p^{\nu})=2$ for every prime $p>2$ and every $\nu \in \N$, $N_1(1,2)=1$, $N_1(1,4)=2$, $N_1(1,2^{\nu})=4$ for every $\nu \ge 3$
(sequence A060594 in \cite{OEIS}). The Dirichlet series representation
\begin{equation} \label{Dirichlet_8_2}
\sum_{r=1}^{\infty} \frac{N_1(1,r)}{r^s}= \frac{\zeta^2(s)}{\zeta(2s)} \left(1-\frac1{2^s}+\frac{2}{2^{2s}}\right)
\end{equation}
shows that estimating the sum $\sum_{r\le x} N_1(1,r)$ is closely related to the squarefree divisor problem. Let
$\tau^{(2)}(n)=2^{\omega(n)}$ denote the number of squarefree divisors of $n$. Then
\begin{equation} \label{squarefree}
\sum_{n=1}^{\infty} \frac{\tau^{(2)}(n)}{n^s}= \frac{\zeta^2(s)}{\zeta(2s)}.
\end{equation}

By this analogy we deduce

\begin{prop} \label{prop_8_2}
\begin{equation*}
\sum_{r\le x} N_1(1,r)= \frac{6}{\pi^2} x\log x +c_1 x + O(x^{1/2}\exp(-c_0(\log x)^{3/5}(\log \log x)^{-1/5})),
\end{equation*}
where $c_0>0$ is a constant and $c_1=\frac{6}{\pi^2}\left(2\gamma-1-\frac{\log 2}{2}- \frac{2\zeta'(2)}{\zeta(2)}\right)$.
If the Riemann hypothesis (RH) is true, then the error term is $O(x^{4/11+\varepsilon})$ for every $\varepsilon >0$.
\end{prop}

\begin{proof} By the identities \eqref{Dirichlet_8_2} and \eqref{squarefree} it follows that for every $r\in \N$,
\begin{equation*}
N_1(1,r)=\sum_{ab=r} \tau^{(2)}(a)h(b),
\end{equation*}
where the multiplicative function $h$ is defined by
\begin{equation*}
h(p^\nu) = \begin{cases}
-1, & \text{ if \ $p=2$, $\nu=1$;} \\
2, & \text{ if \ $p=2$, $\nu=2$;} \\
0, & \text{ otherwise.}
\end{cases}
\end{equation*}

Now the convolution method and the result
\begin{equation*}
\sum_{n\le x} \tau^{(2)}(n)= \frac{6}{\pi^2}x \left( \log x + 2\gamma-1-\frac{2\zeta'(2)}{\zeta(2)}\right) + O(R(x)),
\end{equation*}
where $R(x)\ll x^{1/2}\exp(-c_0(\log x)^{3/5}(\log \log x)^{-1/5})$ (see \cite{SS1970}) conclude the proof. If RH is true, then
the estimate $R(x)\ll x^{4/11+\varepsilon}$ due to Baker \cite{Bak1996} can be used.
\end{proof}

\begin{remark} {\rm See \cite{FMS2010} for asymptotic formulas on the number of solutions of the higher degree congruences
$x^{\ell}\equiv 0$ (mod $n$) and $x^{\ell}\equiv 1$ (mod $n$), respectively, where $\ell \in \N$. The results of our Propositions
\ref{prop_8_1} and \ref{prop_8_2} are better than those of \cite{FMS2010} applied for $\ell =2$.}
\end{remark}

\subsection{The congruence \texorpdfstring{$x^2+y^2\equiv 0$}{x2+y2=0} (mod \texorpdfstring{$r$}{r})}

This is the case $k=2$, $n=0$. $N_2(0,r)$ is the sequence A086933 in \cite{OEIS} and for $r$ odd it is given by \eqref{N_2(0,r)_r_odd}.
Furthermore, $N_2(0,2^{\nu})$ is given by Proposition \ref{prop_6_5}. We deduce

\begin{cor} \label{cor_N_2(0)_prime} For every prime power $p^{\nu}$ ($\nu\in \N$),
\begin{equation*}
N_2(0,p^\nu) = \begin{cases}
p^{\nu} (\nu+1-\nu/p), & \text{ if \ $p\equiv 1$ {\rm (mod $4$)}, $\nu \ge 1$}; \\
p^{\nu}, & \text{ if \ $p\equiv -1$ {\rm (mod $4$)}, $\nu$ is even;} \\
p^{\nu-1}, & \text{ if \ $p\equiv -1$ {\rm (mod $4$)}, $\nu$ is odd;} \\
2^{\nu}, & \text{ if \ $p=2$, $\nu \ge 1$}.
\end{cases}
\end{equation*}
\end{cor}

\begin{cor} \label{cor_N_2_convo} $N_2(0,\DOT)= \id \cdot (\1*\chi)*\mu \chi$, where $\chi$ is the nonprincipal character {\rm (mod $4$)}.
\end{cor}

\begin{proof} From Corollary \ref{cor_N_2(0)_prime} we obtain the Dirichlet series representation
\begin{equation*}
\sum_{r=1}^{\infty} \frac{N_2(0,r)}{r^s} = \zeta(s-1) \prod_{p>2} \left(1-\frac{(-1)^{(p-1)/2}}{p^s} \right)
\left(1-\frac{(-1)^{(p-1)/2}}{p^{s-1}} \right)^{-1}
\end{equation*}
\begin{equation*}
= \zeta(s-1) L(s-1,\chi) L(s,\chi)^{-1},
\end{equation*}
where $L(s,\chi)$ is the Dirichlet series of $\chi$. This gives the result.
\end{proof}

Observe that $4(\1*\chi)(n)= r_2(n)$ is the number of ways $n$ can be written as a sum of two squares, quoted in the Introduction.
This shows that the sum $\sum_{r\le x} N_2(0,r)$ is closely related to the Gauss circle problem.
The next result corresponds to the asymptotic formula due to Huxley \cite{Hux2003}
\begin{equation} \label{Hux}
\sum_{n\le x} r_2(n)= \pi x + O(x^a (\log x)^b),
\end{equation}
where $a=131/416\doteq 0.314903$ and $b=26947/8320$.

\begin{prop} We have
\begin{equation*}
\sum_{r\le x} N_2(0,r)= \frac{\pi}{8G} x^2 +O(x^{a+1} (\log x)^b),
\end{equation*}
where $G$ is the Catalan constant defined by \eqref{Catalan_const}.
\end{prop}

\begin{proof} Since $N_2(0,\DOT)=(\id \cdot \, r_2/4) * (\mu \chi)$, we have
\begin{equation*}
\sum_{r\le x} N_2(0,r)= \frac1{4} \sum_{d\le x} \mu(d)\chi(d) \sum_{n\le x/d} n r_2(n).
\end{equation*}

Now partial summation on \eqref{Hux} and usual estimates give the result.
\end{proof}

\subsection{The congruence \texorpdfstring{$x^2+y^2\equiv 1$}{x2+y2=1} (mod \texorpdfstring{$r$}{r})}

This is the case $k=2$, $n=1$. $N_2(1,r)$ is sequence A060968 in \cite{OEIS}. For every $r$ odd we have by \eqref{eq_4m+2},
\begin{equation*}
N_2(1,r)= r \sum_{d\mid r} (-1)^{(d-1)/2}\ \frac{\mu(d)}{d},
\end{equation*}
and deduce (cf. Corollary \ref{Cor_6_3}).

\begin{cor} For every prime power $p^{\nu}$ ($\nu\in \N$),
\begin{equation*}
N_2(1,p^\nu) = \begin{cases}
p^{\nu} (1-1/p), & \text{ if \ $p \equiv 1$ {\rm (mod $4$)}, $\nu \ge 1$}; \\
p^{\nu} (1+1/p), & \text{ if \ $p \equiv -1$ {\rm (mod $4$)}, $\nu \ge 1$}; \\
2, & \text{ if \ $p=2$, $\nu=1$}; \\
2^{\nu+1}, & \text{ if \ $p=2$, $\nu \ge 2$}.
\end{cases}
\end{equation*}
\end{cor}

\begin{prop} We have
\begin{equation*}
\sum_{r\le x} N_2(1,r)= \frac{5}{8G} x^2 +O(x\log x).
\end{equation*}
\end{prop}

\begin{proof} One has the Dirichlet series representation
\begin{equation*}
\sum_{r=1}^{\infty} \frac{N_2(1,r)}{r^s} =
\zeta(s-1)\left(1+\frac4{2^{2s}}\right) L(s,\chi)^{-1},
\end{equation*}
and the asymptotic formula is obtained by usual elementary arguments.
\end{proof}

\subsection{The congruence \texorpdfstring{$x^2+y^2+z^2\equiv 0$}{x2+y2+z2=0} (mod \texorpdfstring{$r$}{r})}

This is the case $k=3$, $n=0$. $N_3(0,r)$ is the sequence A087687 in
\cite{OEIS}. By Proposition \ref{Prop_5_7} we have for every $r\in \N$ odd,
\begin{equation} \label{N_3(0,r)}
N_3(0,r) = r^2 \sum_{d^2\mid r} \frac{\varphi(d)}{d^2}
\end{equation}
and using also Proposition \ref{Prop_7_1} we deduce

\begin{cor} For every prime power $p^{\nu}$ ($\nu\in \N$),
\begin{equation*}
N_3(0,p^\nu) = \begin{cases}
p^{3\beta-1}(p^{\beta+1}+p^{\beta}-1), & \text{ if \ $p>2$, $\nu=2\beta$ is even;} \\
p^{3\beta-2}(p^{\beta}+p^{\beta-1}-1), & \text{ if \ $p>2$, $\nu=2\beta-1$ is odd;} \\
2^{3\beta}, & \text{ if \ $p=2$, $\nu=2\beta$ is even;} \\
2^{3\beta-1}, & \text{ if \ $p=2$, $\nu=2\beta-1$ is odd.}
\end{cases}
\end{equation*}
\end{cor}

\begin{prop}
\begin{equation*}
\sum_{r\le x} N_3(0,r)= \frac{24\zeta(3)}{\pi^4} x^3 +O(x^2\log x).
\end{equation*}
\end{prop}

\begin{proof} The Dirichlet series of the function $r\mapsto N_3(0,r)$ is
\begin{equation*}
\sum_{r=1}^{\infty} \frac{N_3(0,r)}{r^s} = \zeta(s-2) G(s),
\end{equation*}
where
\begin{equation} \label{G}
G(s)= \frac{\zeta(2s-3)}{\zeta(2s-2)}\frac{2^{2s}-16}{2^{2s}-4}
\end{equation}
is the Dirichlet series of the multiplicative function $g$ given by
\begin{equation*}
g(p^{\nu}) = \begin{cases}
p^{3\beta-1}(p-1), & \text{ if \ $p>2$, $\nu=2\beta \ge 2$;} \\
-2^{3\beta}, & \text{ if \ $p=2$, $\nu=2\beta \ge 2$;} \\
0, & \text{ if \ $p\ge 2$, $\nu=2\beta -1\ge 1$.}
\end{cases}
\end{equation*}

Therefore, $N_3(0,\DOT)=\id_2* g$ and obtain
\begin{align}
\sum_{r\le x} N_3(0,r) & = \sum_{d\le x} g(d)\left(\frac{x^3}{3d^3}+ O(\frac{x^2}{d^2})\right) \nonumber \\
& =\frac{x^3}{3}G(3) + O\left(x^3\sum_{d>x} \frac{|g(d)|}{d^3}\right) +  O\left(x^2\sum_{d\le x} \frac{|g(d)|}{d^2}\right).
\label{sum_N_3}
\end{align}

Here a direct computation shows that
\begin{equation} \label{log}
\sum_{d\le x} \frac{|g(d)|}{d^2} \le \prod_{p\le x} \sum_{\nu=0}^{\infty} \frac{|g(p^\nu)|}{p^{\nu}}  \ll \prod_p \left(1+\frac1{p}\right)
\ll \log x
\end{equation}
by Mertens' theorem.

Furthermore, by \eqref{G}, $g(n)=\sum_{ab^2=n} h(a)b^3$, where the Dirichlet series of the function $h$ is absolutely convergent
for $\Re s>3/2$. Hence
\begin{equation*}
\sum_{n\le x} h(n)= c_2x^2+  O(x^{3/2+\varepsilon})
\end{equation*}
with a certain constant $c_2$, and by partial summation we deduce that

\begin{equation} \label{1_per_x}
\sum_{d>x} \frac{|g(d)|}{d^3} \ll \frac1{x}.
\end{equation}

Now the result follows from \eqref{sum_N_3}, \eqref{log} and \eqref{1_per_x}.
\end{proof}

\subsection{The congruence \texorpdfstring{$x^2+y^2+z^2\equiv 1$}{x2+y2+z2=1} (mod \texorpdfstring{$r$}{r})}

$N_3(1,r)$ is the sequence A087784 in \cite{OEIS}. Using Corollary \ref{cor_5_4}
and Proposition \ref{Prop_7_2} we have

\begin{cor} For every prime power $p^{\nu}$ ($\nu\in \N$),
\begin{equation*}
N_3(1,p^\nu) = \begin{cases}
p^{2\nu} (1+1/p), & \text{ if \ $p \equiv 1$ {\rm (mod $4$)}, $\nu \ge 1$}; \\
p^{2\nu} (1-1/p), & \text{ if \ $p \equiv -1$ {\rm (mod $4$)}, $\nu \ge 1$}; \\
4, & \text{ if \ $p=2$, $\nu=1$}; \\
3\cdot 2^{2\nu-1}, & \text{ if \ $p=2$, $\nu \ge 2$}.
\end{cases}
\end{equation*}
\end{cor}

\begin{prop} We have
\begin{equation*}
\sum_{r\le x} N_3(1,r)= \frac{36G}{\pi^4} x^3 +O(x^2\log x).
\end{equation*}
\end{prop}

\begin{proof} Now the corresponding Dirichlet series is
\begin{equation*}
\sum_{r=1}^{\infty} \frac{N_3(1,r)}{r^s} =
\zeta(s-2)\left(1+\frac{8}{2^{2s}}\right) \prod_{p\equiv 1 \text{
(mod $4$)}} \left(1+\frac1{p^{s-1}} \right) \prod_{p\equiv -1 \text{
(mod $4$)}} \left(1-\frac1{p^{s-1}} \right).
\end{equation*}

Hence, $N_3(1,\DOT)= \id_2 * f$, where $f$ is the multiplicative function defined for prime powers $p^{\nu}$ by
\begin{equation*}
f(p^\nu) = \begin{cases}
p, & \text{ if \ $p \equiv 1$ {\rm (mod $4$)}, $\nu = 1$}; \\
-p, & \text{ if \ $p \equiv -1$ {\rm (mod $4$)}, $\nu = 1$}; \\
8, & \text{ if \ $p=2$, $\nu=2$}; \\
0, &  \text{ otherwise,}
\end{cases}
\end{equation*}
and the given asymptotic formula is obtained by the convolution method.
\end{proof}

\subsection{The congruence \texorpdfstring{$x^2+y^2+z^2+t^2\equiv 0$}{x2+y2+z2+t2=0} (mod \texorpdfstring{$r$}{r})}

This is the case $k=4$, $n=0$ (sequence A240547 in \cite{OEIS}). For
every $r$ odd,
\begin{equation*}
N_4(0,r)= r^3 \sum_{d\mid r} \frac{\varphi(d)}{d^2}
\end{equation*}
by Corollary \ref{cor_4_2} and using also Proposition \ref{prop_6_4}
we conclude

\begin{cor} For every prime power $p^{\nu}$ ($\nu\in \N$),
\begin{equation*}
N_4(0,p^\nu) = \begin{cases}
p^{2\nu-1} (p^{\nu+1}+p^{\nu}-1), & \text{ if \ $p>2$, $\nu \ge 1$}; \\
2^{2\nu+1}, & \text{ if \ $p=2$, $\nu \ge 1$}.
\end{cases}
\end{equation*}
\end{cor}

\begin{prop} We have
\begin{equation*}
\sum_{r\le x} N_4(0,r)= \frac{5\pi^2}{168\zeta(3)} x^4 +O(x^3 \log x).
\end{equation*}
\end{prop}

\begin{proof} The corresponding Dirichlet series is
\begin{equation*}
\sum_{r=1}^{\infty} \frac{N_4(0,r)}{r^s} = \zeta(s-2)\zeta(s-3)
\left(1-\frac4{2^s}-\frac{32}{2^{2s}} \right) \prod_{p>2}
\left(1-\frac1{p^{s-1}}\right).
\end{equation*}
\end{proof}

\subsection{The congruence \texorpdfstring{$x^2+y^2+z^2+t^2\equiv 1$}{x2+y2+z2+t2=1} (mod \texorpdfstring{$r$}{r})}

This is the case $k=4$, $n=1$. $N_4(1,r)$ is sequence A208895 in \cite{OEIS}. By the identity
\eqref{form_spec_Jordan} giving its values for $r$ odd and by Corollary \ref{cor_6_2} we obtain

\begin{cor} For every prime power $p^{\nu}$ ($\nu\in \N$),
\begin{equation*}
N_4(1,p^\nu) = \begin{cases}
p^{3\nu} (1-1/p^2), & \text{ if \ $p>2$, $\nu \ge 1$}; \\
8^{\nu}, & \text{ if \ $p=2$, $\nu \ge 1$}.
\end{cases}
\end{equation*}
\end{cor}

\begin{prop} We have
\begin{equation*}
\sum_{r\le x} N_4(1,r)= \frac{2}{7\zeta(3)} x^4 +O(x^3).
\end{equation*}
\end{prop}

\begin{proof} Here
\begin{equation*}
\sum_{r=1}^{\infty} \frac{N_4(1,r)}{r^s} = \zeta(s-3) \prod_{p>2}
\left(1-\frac1{p^{s-1}}\right).
\end{equation*}
\end{proof}

Finally, we deal with two special cases corresponding to $\a \ne (1,\ldots,1)$.

\subsection{The congruence \texorpdfstring{$x^2-y^2\equiv 1$}{x2-y2=1} (mod \texorpdfstring{$r$}{r})}

Here $k=2$, $n=1$, $\a=(1,-1)$. $N_2(1,r,(1,-1))$ is sequence
A062570 in \cite{OEIS}. Corollary  \ref{cor_4_8} tells us that for
every $r\in \N$ odd, $N_3(1,r,(1,-1))= \varphi(r)$. Furthermore,
from Proposition \ref{Prop_3_1} one can deduce, similar to the proof
of Proposition \ref{Prop_6_1} that for every $\nu \in \N$,
$N_3(1,2^{\nu},(1,-1))= 2^{\nu}$. Thus,

\begin{cor} For every $r\in \N$ one has
\begin{equation*}
N_3(1,r,(1,-1))= \varphi(2r).
\end{equation*}
\end{cor}

\subsection{The congruence \texorpdfstring{$x^2+y^2\equiv z^2$}{x2+y2=z2} (mod \texorpdfstring{$r$}{r})}

This Phythagorean congruence is obtained for $k=3$, $n=0$,
$a_1=a_2=1$, $a_3=-1$.  $N_3(0,r,(1,1,-1))$ is sequence A062775 in
\cite{OEIS}. Proposition \ref{Prop_5_7} shows that for every $r\in
\N$ odd, $N_3(0,r,(1,1,-1))= N_3(0,r)$ given by \eqref{N_3(0,r)}.
From Proposition \ref{Prop_3_1} one can deduce that for every $\nu
\in \N$,
\begin{equation*}
N_3(0,2^{\nu},(1,1,-1))= 2^{2\nu} \left(2- \frac1{2^{\lfloor \nu/2 \rfloor}} \right).
\end{equation*}

Consequently,
\begin{cor}
\begin{equation*}
N_3(0,p^\nu,(1,1,-1)) =
\begin{cases}
p^{3\beta-1}(p^{\beta+1}+p^{\beta}-1), & \text{ if \ $p>2$, $\nu=2\beta$ is even;} \\
p^{3\beta-2}(p^{\beta}+p^{\beta-1}-1), & \text{ if \ $p>2$, $\nu=2\beta-1$ is odd;} \\
2^{3\beta}(2^{\beta+1}-1), & \text{ if \ $p=2$, $\nu=2\beta$ is even;} \\
2^{3\beta-1}(2^{\beta}-1), & \text{ if \ $p=2$, $\nu=2\beta-1$ is odd.}
\end{cases}
\end{equation*}
\end{cor}


\section{Acknowledgement} The author gratefully acknowledges support
from the Austrian Science Fund (FWF) under the project Nr.
M1376-N18. The author thanks Olivier Bordell\`{e}s, Steven Finch and Wenguang Zhai for their valuable comments on
the asymptotic formulas included in Section \ref{Sect_spec_cases}. The author is thankful to the referee for some corrections
and remarks that have improved the paper.


\noindent L. T\'oth \\
Department of Mathematics, University of P\'ecs \\ Ifj\'us\'ag \'utja 6,
H-7624 P\'ecs, Hungary \\ and \\
Institute of Mathematics, Universit\"at f\"ur Bodenkultur \\
Gregor Mendel-Stra{\ss}e 33, A-1180 Vienna, Austria  \\
E-mail: ltoth@gamma.ttk.pte.hu


\begin{thebibliography}{99}

\bibitem{Apo1976} T.~M.~Apostol, {\it Introduction to Analytic Number Theory},
Sprin\-ger, 1976.

\bibitem{Bac1898} P.~Bachmann, {\it Zahlentheorie}, vol. 4: Arithmetik der quadratischen Formen, Leipzig, 1898.

\bibitem{Bak1996} R.~C.~Baker, The square-free divisor problem II, {\it Quart. J. Math. (Oxford)} (2) {\bf 47} (1996),
133--146.

\bibitem{BEW1998} B.~C.~Berndt, R.~J.~Evans, and K.~S.~Williams, {\it Gauss and Jacobi Sums},
Canadian Mathematical Society Series of Monographs and Advanced Texts, New York, NY: John Wiley \& Sons, 1998.

\bibitem{Bor2012} O.~Bordell\`{e}s, {\it Arithmetic Tales}, Springer,
2012.

\bibitem{Coh1954Duke} E.~Cohen, Rings of arithmetic functions. II: The
number of solutions of quadratic congruences, {\it Duke Math. J.}
{\bf 21} (1954), 9--28.

\bibitem{Coh1963Nachr} E.~Cohen, A generalization of Axer's theorem and some of its applications,
{\it Math. Nachr.} {\bf 27} (1963/1964), 163--177.

\bibitem{Coh1964Collect} E.~Cohen, Arithmetical notes, VII: Some classes of
even functions (mod $r$), {\it Collect. Math.} {\bf 16} (1964),
81--87.

\bibitem{Coh1966AMM} E.~Cohen, Quadratic congruences with an odd number of
summands, {\it Amer. Math. Monthly} {\bf 73} (1966), 138--143.

\bibitem{Dic1966} L.~E.~Dickson, {\it History of the Theory of Numbers}, vol. II, reprinted by Chelsea, New York, 1966.

\bibitem{FMS2010} S.~Finch, G.~Martin, and P.~Sebah, Roots of unity and nullity modulo $n$, {\it Proc. Amer. Math. Soc.}
{\bf 138} (2010), 2729--2743.

\bibitem{Gir2013} K.~Girstmair, The equation $x^2+ \overline{m}y^2 =\overline{k}$ in $\Z/p\Z$,
{\it Amer. Math. Monthly} {\bf 120} (2013), 546--552.

\bibitem{Gro1985} E.~Grosswald, {\it Representations of Integers as Sums of
Squares}, Springer, 1985.

\bibitem{HarWri2008} G.~H,~Hardy and E.~M.~Wright, {\it An Introduction to the Theory of Numbers}, Sixth Edition, Edited and revised by
D.~R.~Heath-Brown and J.~H.~Silverman, Oxford University Press, 2008.

\bibitem{Hua1982} L.~K.~Hua, {\it Introduction to Number Theory}, Springer,
1982.

\bibitem{Hux2003} M.~N.~Huxley, Exponential sums and
lattice points III., {\it Proc. London Math. Soc.} {\bf 87} (2003),
591--609.

\bibitem{IreRos1990} K.~Ireland and M.~Rosen, {\it A Classical Introduction to Modern Number Theory}, 2nd ed.,
Graduate Texts in Mathematics 84, Springer, 1990.

\bibitem{Min1911} H.~Minkowski, {\it Gesammelte Abhandlungen}, Leipzig, 1911.

\bibitem{Nar1983} W.~Narkiewicz, {\it Number Theory}, World Scientific, Singapore, 1983.

\bibitem{Nat2000} M.~B.~Nathanson, {\it Elementary Methods in Number
Theory}, Graduate Texts in Mathematics 195, Springer, 2000.

\bibitem{RabSha1986} M.~O.~Rabin and J.~O.~Shallit, Randomized algorithms in number theory, {\it Comm. in Pure and Applied Math.}
{\bf 39} (supplement) (1986), S239--S256.

\bibitem{OEIS} N.~J.~A.~Sloane, The On-Line Encyclopedia of Integer Sequences.
\url{http://oeis.org}

\bibitem{SS1970} D.~Suryanarayana and V.~Siva Rama Prasad, The number of $k$-free divisors of an integer, {\it Acta Arith.} {\bf 17}
(1971), 345--354.

\bibitem{TotHau2011} L.~T\'oth and P.~Haukkanen, The discrete Fourier transform of $r$-even functions, {\it Acta Univ. Sapientiae, Math.} {\bf 3}
(2011), 5--25.

\end{thebibliography}
\end{document}